\documentclass{amsart}

\usepackage{amsmath,amssymb,amsthm,amsfonts,mathrsfs}
\usepackage{fullpage}
\usepackage{relsize}
\usepackage{colonequals}

\usepackage[colorlinks=true,
linkcolor=blue,
anchorcolor=blue,
citecolor=red
]{hyperref}

\allowdisplaybreaks 
\newtheorem{theorem}{Theorem}[section]

\newtheorem{corollary}[theorem]{Corollary}
\newtheorem{proposition}[theorem]{Proposition}

\theoremstyle{definition}

\theoremstyle{remark}
\newtheorem{remark}[theorem]{Remark}
\newcommand{\N}{\mathbb{N}}
\newcommand{\Z}{\mathbb{Z}}

\newcommand{\R}{\mathbb{R}}
\newcommand{\C}{\mathbb{C}}

\newcommand{\ve}{\varepsilon}
\newcommand{\on}{\operatorname}
\renewcommand{\mod}[1]{\,\on{mod}#1}
\newcommand{\of}[1]{\left(#1\right)}
\newcommand{\set}[1]{\left\{#1\right\}}
\newcommand{\AVG}{\frac{1}{N} \sum_{n=1}^{N}}
\newcommand{\BEu}[1]{\underset{#1}{\mathlarger{\mathlarger{\mathbb{E}}}^{~}}\,}

\author[B. Wang]{Biao Wang}
\address{School of Mathematics and Statistics, Yunnan University, Kunming, Yunnan 650091, China}
\email{bwang@ynu.edu.cn}

\author{Shaoyun Yi}
\address{School of Mathematical Sciences, Xiamen University, Xiamen, Fujian 361005, China}
\email{yishaoyun926@xmu.edu.cn}

\date{\today}

\makeatletter
\@namedef{subjclassname@2020}{\textup{}2020 Mathematics Subject Classification}
\makeatother

\title{The prime number theorem over integers of power-free polynomial values}

\begin{document}
	
\begin{abstract}
Let $f(x)\in \Z[x]$ be an irreducible polynomial of degree $d\ge 1$. Let $k\ge2$ be an integer. The number of integers $n$ such that $f(n)$ is $k$-free is widely studied in the literature. In principle, one expects that $f(n)$ is $k$-free infinitely often, if $f$ has no fixed $k$-th power divisor. In 2022, Bergelson and Richter established a new dynamical generalization of the prime number theorem (PNT).   Inspired by their work, one may expect that this generalization of the PNT also holds over integers of power-free polynomial values. In this note, we establish such variants of Bergelson and Richter's theorem for several polynomials studied by Estermann, Hooley, Heath-Brown, Booker and Browning.
\end{abstract}

\maketitle

\numberwithin{equation}{section}

\section{Introduction}

For any $n\in\N$, let $\Omega(n)$ be the number of prime divisors of $n$ counted with multiplicity.  Let $\lambda(n)=(-1)^{\Omega(n)}$ be the Liouville function. It is well-known (e.g., \cite{Landau1909}) that the prime number theorem (PNT) is equivalent to the assertion that
\begin{equation}\label{pnt_lambda} 
			\lim_{N \to \infty} \frac1N\sum_{n=1}^N \lambda(n)=0.
\end{equation}

Let $(X,\mu, T)$ be a uniquely ergodic topological dynamical system. In 2022, Bergelson and Richter \cite{BergelsonRichter2022} established a new dynamical generalization of the PNT, showing that  
\begin{equation}\label{eqn_BR2022thmA}
		\lim_{N\to\infty}\frac1N\sum_{n=1}^N g(T^{\Omega(n)}x)=\int_X g \,d\mu
\end{equation}
holds for any $g\in C(X)$ and $x\in X$, where $C(X)$ denotes the set of continuous functions
defined on $X$.   The equivalent form \eqref{pnt_lambda}  of the PNT is recovered from  \eqref{eqn_BR2022thmA} by taking $(X,T)$ as rotation on two points. 

Let $\boldsymbol{\mu}(n)$ be the M\"obius function, which is defined to be $1$ if $n = 1$,  $(-1)^r$ if $n$ is the product of $r$ distinct primes, and zero otherwise. For an integer $k\ge2$,  a natural number $n$ is said to be \emph{$k$-free} if there is no prime $p$ such that $p^k\mid n$. Then $\boldsymbol{\mu}^2(n)$ is the indicator function of squarefree numbers. If the numbers $n$ are restricted to be squarefree in \eqref{eqn_BR2022thmA}, then Bergelson and Richter \cite{BergelsonRichter2022} also proved that
\begin{equation}\label{eqn_BR2022cor1.8}
	\lim_{N \to \infty} \frac1N\sum_{\substack{1\le n \le N\\ n\,\text{squarefree}}} g(T^{\Omega(n)}x)=\frac6{\pi^2}\int_X g \,d\mu 
\end{equation}
holds for any $g\in C(X)$ and $x\in X$. This is another  dynamical generalization of the PNT. Taking $(X,T)$ as rotation on two points in \eqref{eqn_BR2022cor1.8} gives the following well-known
equivalent form of the PNT:
$$	
\lim_{N \to \infty} \frac1N\sum_{n=1}^N \boldsymbol{\mu}(n)=0.
$$

Motivated by \eqref{eqn_BR2022cor1.8}, the first author \cite{Wang2025bams} recently proved the following variant of Bergelson and Richter's theorem along twins of squarefree numbers: 
	\begin{equation}\label{Wang2025bams}
	\lim_{N\to\infty}\frac1N \sum_{\substack{1\leq n \leq N\\ n,\, n+1 \text{ squarefree}}} g(T^{\Omega(n)}x)= \prod_{p}\Big(1-\frac{2}{p^2}\Big) \int_X g \,d\mu
\end{equation}
holds for any $g\in C(X)$ and $x\in X$. This is related to counting the number of integers $n$ such that $n^2+n$ is squarefree. Indeed, in 1932, Carlitz \cite{Carlitz1932} showed that
\begin{equation}\label{Carlitz}
	\sum_{n\leq N}\boldsymbol{\mu}^2(n^2+n)=\prod_{p}\Big(1-\frac2{p^2}\Big) N+ O(N^{2/3+\ve})
\end{equation}
for any $\ve>0$, where the implied constant depends on $\ve$. In 1984, Heath-Brown \cite{HB1984} improved the error term of \eqref{Carlitz} to $O(N^{7/12}\log^7N)$ by using the square sieve. 
It was further improved by Reuss \cite{Reuss2015} to $O(N^{0.578+\ve})$ by using the approximate determinant method introduced by Heath-Brown in \cite{HB2009, HB2012}. 

In general, a number of results  have been established on counting the number of integers with power-free polynomial values in the literature. For instance, in 1931, Estermann \cite{Estermann1931} showed that
\begin{equation}\label{Estermann1931}
	\sum_{n\leq N} \boldsymbol{\mu}^2(n^2+1)=c_0N + O(N^{2/3}\log N),
\end{equation}
where $c_0$ is an absolute constant. The error term of \eqref{Estermann1931} was improved to $O(N^{7/12+\ve})$ by Heath-Brown \cite{HB2012} in 2012. Moreover, in \cite{HB2013} he showed that for an irreducible polynomial $x^d + c \in \Z[x]$ of degree $d$,  if $k \geq (5d + 3)/9$, then there is a positive constant $\delta(d)$ such that
\begin{equation}\label{HB2013}
	\#\left\{n\leq N \colon n^d+c \text{ is }k\text{-free}\right\} = \prod_{p}\Big(1-\frac{\rho(p^k)}{p^k}\Big) \cdot N + O(N^{1-\delta(d)}),
\end{equation}
where $\rho(q)$ denotes the number of solutions to $\nu^d+c\equiv 0 \mod{q}$. Browning showed in \cite{Browning2011} that for any general irreducible polynomial $f\in \Z[x]$  of degree $d\ge3$, that
\begin{equation}\label{Browning2011}
	\#\set{n\leq N \colon f(n) \text{ is } k\text{-free}} \sim \prod_p\Big(1-\frac{\rho_f(p^k)}{p^k}\Big) \cdot N
\end{equation}
holds, provided that $k \geq (3d+1)/4$.  Here $\rho_f(q)$ denotes the number of solutions to $f(\nu)\equiv 0\mod{q}$, i.e.,
\begin{equation}\label{dfn_rho}
	\rho_f(q)\colonequals \#\set{ \nu \mod{q}\colon f(\nu) \equiv 0 \mod{q}}.
\end{equation}
In \cite{Xiao2017}, Xiao gave another proof of Browning's result \eqref{Browning2011}. For results on establishing \eqref{Browning2011} for other polynomials, see \cite{Ricci1933, Hooley1967, Nair1976, Nair1979, FriedlanderIwaniec2010, Reuss2015blms, BrowningShparlinski2024} and so on. 

In principle, one expects that \eqref{Browning2011} holds for any irreducible polynomial $f(x)$ with no fixed $k$-th power divisor, i.e., there is no prime $p$ such that $p^k\mid f(n)$ for all $n\in\Z$. Motivated by \eqref{eqn_BR2022cor1.8}-\eqref{Browning2011}, we expect that Bergelson and Richter's theorem \eqref{eqn_BR2022thmA} holds for power-free values of polynomials. That is, let $k\ge2$, and let $f(x)\in \Z[x]$ be an irreducible polynomial  with no fixed $k$-th power divisor, then in a uniquely ergodic topological dynamical system $(X,\mu, T)$, we expect that
	 	\begin{equation}\label{conjecture}
		\lim_{N\to\infty}\frac1N\sum_{1\leq n\leq N} \nu_k(f(n)) g(T^{\Omega(n)}x) = \prod_p\Big(1-\frac{\rho_f(p^k)}{p^k}\Big)\cdot \int_X g\,d\mu
	\end{equation}
holds for any $g\in C(X)$ and $x\in X$, where $\nu_k(n)$ is the indicator function of $k$-free numbers. In this article, building on the work of Heath-Brown \cite{HB2013} and Browning \cite{Browning2011}, we will establish \eqref{conjecture} for the polynomials in \eqref{HB2013} and \eqref{Browning2011}.

\begin{theorem}\label{mainthm}
	Let $k,d$ be two integers with $k\ge2$. Let $f(x)\in \Z[x]$ be an irreducible polynomial of degree $d\ge 1$.  Let $(X,\mu, T)$ be a uniquely ergodic topological dynamical system.   If $f(x)=x^d+c$ with $k \geq (5d + 3)/9$ and $c\in\Z$, or $k\ge 3d/4+1/4$ for $d\ge3$, then we have
	 	\begin{equation}\label{mainthm_eqn}
		\lim_{N\to\infty}\frac1N\sum_{1\leq n\leq N} \nu_k(f(n)) g(T^{\Omega(n)}x) = \prod_p\Big(1-\frac{\rho_f(p^k)}{p^k}\Big)\cdot \int_X g\,d\mu
	\end{equation}
for any $g\in C(X)$ and $x\in X$, where $\rho_f(p^k)$ is defined in \eqref{dfn_rho}.
\end{theorem}

Taking $f(n)=n^2+1$ and $k=2$ in \eqref{mainthm_eqn}, we have $\rho_f(4)=0$ for $p=2$, $\rho_f(p^2)=2$ for $p\equiv 1\mod{4}$ and $\rho_f(p^2)=0$ for other odd $p$. By Theorem~\ref{mainthm}, we obtain the following variant of Bergelson and Richter's theorem related to Estermann's result \eqref{Estermann1931}.

\begin{corollary}
Let $(X,\mu, T)$ be a uniquely ergodic topological dynamical system. Then we have
	 	\begin{equation}
		\lim_{N\to\infty}\frac1N\sum_{1\leq n\leq N} \boldsymbol{\mu}^2(n^2+1) g(T^{\Omega(n)}x) = \prod_{p\equiv 1 \mod{4}}\Big(1-\frac{2}{p^2}\Big)\cdot \int_X g\,d\mu
	\end{equation}
for any $g\in C(X)$ and $x\in X$.

\end{corollary}

In \cite{HB2013, Browning2011}, to prove \eqref{HB2013} and \eqref{Browning2011}, Heath-Brown and Browning estimated the number of integers $n$ such that $d^k\mid f(n)$ for large moduli $d$. These help us prove Theorem~\ref{mainthm} in an elementary way in the following two sections. In Section~\ref{sec_reducible}, we will show that \eqref{conjecture} also holds for a kind of reducible polynomials considered by Booker and Browning in \cite{BookerBrowning2016}. In Section~\ref{sec_ABC}, under the assumption that the ABC conjecture is true, we will use Murty and Pasten’s results in \cite{MurtyPasten2014} to obtain \eqref{conjecture}  for general polynomials.

\section{A preliminary estimation}\label{sec_pre}

Let $f(x)\in \Z[x]$ be an irreducible polynomial of degree $d\ge 1$. To estimate the partial summation in \eqref{conjecture}, we use the following fact concerning $k$-free numbers
\[
\sum_{d^k\mid f(n)}\boldsymbol{\mu}(d) = 
\begin{cases}
	1, & \text{if } f(n) \text{ is } k\text{-free},\\
	0, & \text{otherwise}.
\end{cases}
\]
Let  $a\colon \N\to \C$ be a bounded arithmetic function. Then we divide the following summation into two parts
	\begin{align}
		& \quad\sum_{1\leq n\leq N}\nu_k(f(n))a(n) = \sum_{1\leq n\leq N} a(n) \sum_{d^k\mid f(n)} \boldsymbol{\mu}(d) = \sum_{d}\boldsymbol{\mu}(d) \sum_{\substack{1\leq n\leq N \\ d^k\mid f(n)}} a(n) \nonumber\\
		&= \sum_{d \leq Y }\boldsymbol{\mu}(d) \sum_{\substack{1\leq n\leq N \\ d^k\mid f(n)}} a(n) + \sum_{d > Y}\boldsymbol{\mu}(d) \sum_{\substack{1\leq n\leq N \\ d^k\mid f(n)}} a(n) \colonequals S_1+S_2, \label{eqn_summation}
	\end{align}
where $1\leq Y \leq |f(N)|^{1/k}$ is a parameter to be determined according to the polynomials we consider. 

In this section, we will give an estimation for $S_1$ with small modulus $d$, which builds a connection between the partial summation in \eqref{conjecture} and the partial summation of $a(n)$. In the next section, we will show that \eqref{conjecture} holds if the part $S_2$ with large modulus is small.  If $B$ is a finite nonempty set, we define 
$$\BEu{x\in B}b(x)\colonequals\frac1{|B|}\sum_{x\in B}b(x)$$
for any function $b\colon B\to\C$ on $B$.

\begin{proposition}\label{keyprop}

Let $k\ge2$ be an integer. Let $f(x)\in \Z[x]$ be an irreducible polynomial of degree $d\ge 1$. Let  $a\colon \N\to \C$ be a bounded arithmetic function. Then for any $\ve>0$ and $1\leq D\leq Y \leq |f(N)|^{1/k}$, we have
\begin{equation}\label{keyprop_eqn}
	\sum_{1\leq n\leq N}\nu_k(f(n))a(n) = N\sum_{d \leq D } \frac{\boldsymbol{\mu}(d)}{d^k} \sum_{\substack{\nu \mod{d^k} \\ f(\nu) \equiv 0 \mod{d^k}} } \BEu{\substack{ 1\leq n\leq N \\ n \equiv \nu \mod{d^k} }} a(n) +  O(ND^{-k+1+\ve}+Y^{1+\ve})+O\big(E_f(Y,N)\big), 
\end{equation}
where
\begin{equation}\label{dfn_E_f}
	E_f(Y,N)=\sum_{\substack{d>Y\\ \boldsymbol{\mu}^2(d)=1}} \sum_{\substack{1\leq n\leq N \\ d^k\mid f(n)}} 1.
\end{equation}
The implied constants in \eqref{keyprop_eqn} depend only on $\ve$, $f$ and the supnorm of $a(n)$.
\end{proposition}

\begin{proof}

By \eqref{eqn_summation}, we have 
\begin{equation}
	\sum_{1\leq n\leq N}\nu_k(f(n))a(n) = S_1+S_2,
	\end{equation}
where
\begin{equation}
	S_1=\sum_{d \leq Y }\boldsymbol{\mu}(d) \sum_{\substack{1\leq n\leq N \\ d^k\mid f(n)}} a(n)\quad \text{ and }\quad   S_2=\sum_{d > Y}\boldsymbol{\mu}(d) \sum_{\substack{1\leq n\leq N \\ d^k\mid f(n)}} a(n).
\end{equation}

For $S_1$, we have
\begin{align}
	S_1&= \sum_{d \leq Y }\boldsymbol{\mu}(d) \sum_{\substack{\nu \mod{d^k} \\ f(\nu) \equiv 0 \mod{d^k}} }  \sum_{\substack{ 1\leq n\leq N \\ n \equiv \nu \mod{d^k} }} a(n)\nonumber\\
	&= \sum_{d \leq Y }\boldsymbol{\mu}(d) \sum_{\substack{\nu \mod{d^k} \\ f(\nu) \equiv 0 \mod{d^k}} }  \Big(\frac{N}{d^k}+O(1)\Big)\BEu{\substack{ 1\leq n\leq N \\ n \equiv \nu \mod{d^k} }} a(n)  \nonumber\\
	&= N \sum_{d \leq Y } \frac{\boldsymbol{\mu}(d)}{d^k} \sum_{\substack{\nu \mod{d^k} \\ f(\nu) \equiv 0 \mod{d^k}} } \BEu{\substack{ 1\leq n\leq N \\ n \equiv \nu \mod{d^k} }} a(n) + O\of{\sum_{d\leq Y} |\boldsymbol{\mu}(d)|\rho_f(d^k)}.\label{S1_1}
\end{align}

By the Chinese remainder theorem, $\rho_f(q)$ is a multiplicative function. By \cite[Theorem~54]{Nagell1964}, we have $\rho_f(d^k)\ll d^\ve$ for any squarefree $d\in\N$. This implies that
 \begin{equation}\label{S1_2}
 	\sum_{d\leq Y} |\boldsymbol{\mu}(d)|\rho_f(d^k) \ll Y^{1+\ve}.
 \end{equation}
 Here and thereafter mentioned implied constants in the estimates involving $\ve$ depend on $\ve$. 
  
 For any $1\leq D\leq Y$, by $|\boldsymbol{\mu}(d)|\rho_f(d^k)\ll d^\ve$  we have 
 \begin{equation}\label{S1_3}
 	\Big|\sum_{D< d \leq Y } \frac{\boldsymbol{\mu}(d)}{d^k} \sum_{\substack{\nu \mod{d^k} \\ f(\nu) \equiv 0 \mod{d^k}} } \BEu{\substack{ 1\leq n\leq N \\ n \equiv \nu \mod{d^k} }} a(n) \Big| \ll  \sum_{d>D} \frac{|\boldsymbol{\mu}(d)|\rho(d^k)}{d^k} \ll \sum_{d>D} \frac{1}{d^{k-\ve}} \ll D^{-k+1+\ve}.
 \end{equation}
 
Then combining \eqref{S1_1}-\eqref{S1_3} gives us that 
 \begin{equation}\label{S1}
 	S_1=N \sum_{d \leq D } \frac{\boldsymbol{\mu}(d)}{d^k} \sum_{\substack{\nu \mod{d^k} \\ f(\nu) \equiv 0 \mod{d^k}} } \BEu{\substack{ 1\leq n\leq N \\ n \equiv \nu \mod{d^k} }} a(n) + O(ND^{-k+1+\ve}+Y^{1+\ve}).
 \end{equation}
 
For $S_2$, we use the following trivial inequality 
 \begin{equation}\label{S2}
 	|S_2|\ll \sum_{\substack{d>Y\\ \boldsymbol{\mu}^2(d)=1}} \sum_{\substack{1\leq n\leq N \\ d^k\mid f(n)}} 1 = E_f(Y,N).
 \end{equation}
Thus, \eqref{keyprop_eqn} follows by \eqref{S1} and \eqref{S2}.
\end{proof}

\section{Proof of Theorem~\ref{mainthm}} \label{sec_proof}

In this section, we will establish \eqref{conjecture} if the term $E_f(Y,N)$ defined in \eqref{dfn_E_f} is small. To prove \eqref{conjecture}, we cite the following dynamical generalization of the PNT for arithmetic progressions showed by Bergelson and Richter in \cite{BergelsonRichter2022}.

\begin{theorem}[{\cite[Corollary~1.16]{BergelsonRichter2022}}]\label{thm_BR_cor_1.16}
	Let $(X, \mu, T)$ be a uniquely ergodic topological dynamical system. Then we have
\begin{equation}\label{eqn_maincor_BR} 
	\lim_{N\to\infty}\frac1N \sum_{1\leq n \leq N} g(T^{\Omega(mn+r)}x)= \int_X g \,d\mu 
\end{equation}
for any $g\in C(X)$, $x\in X$, $m\in \N$ and $r\in \set{0,1\dots,m-1}$.
\end{theorem}

Now, we show the following theorem, which will imply Theorem~\ref{mainthm} by the work of Heath-Brown \cite{HB2013} and Browning \cite{Browning2011}.

\begin{theorem}\label{thm_general}
Let $k\ge2$ be an integer. Let $f(x)\in \Z[x]$ be an irreducible polynomial of degree $d\ge 1$. Let $(X,\mu, T)$ be a uniquely ergodic topological dynamical system. Let $E_f(Y,N)$ be defined in \eqref{dfn_E_f}. If there exist some positive constants $\delta, \Delta>0$ such that 
\begin{equation}\label{thm_general_cond}
	E_f(N^{1-\delta},N)\ll N^{1-\Delta},
\end{equation}
then we have
	 	\begin{equation}\label{thm_general_eqn}
		\lim_{N\to\infty}\frac1N\sum_{1\leq n\leq N} \nu_k(f(n)) g(T^{\Omega(n)}x) = \prod_p\Big(1-\frac{\rho_f(p^k)}{p^k}\Big)\cdot \int_X g\,d\mu
	\end{equation}
for any $g\in C(X)$ and $x\in X$, where $\rho_f(p^k)$ is defined in \eqref{dfn_rho}.
\end{theorem}

\begin{proof} 

Put $a(n)=g(T^{\Omega(n)}x)$ and $\alpha=\int_X g\,d\mu$ in \eqref{thm_general_eqn}. Then by Theorem~\ref{thm_BR_cor_1.16}, we have
\begin{equation}\label{eqn_BR_pntap}
	\lim_{N\to\infty}\frac1N \sum_{1\leq n \leq N} a(mn+r)=\alpha
\end{equation}
for any $m\in \N$ and $r\in \set{0,1\dots,m-1}$. Clearly, $a(n)$ is bounded. We may assume that  $|a(n)|\leq 1$ for all $n$. 

By Proposition~\ref{keyprop}, taking $Y=N^{1-\delta}$ and $\ve=\delta$ in \eqref{keyprop_eqn} we get that
\begin{equation}\label{thm_general_pf1}
		\frac1N\sum_{1\leq n\leq N}\nu_k(f(n))a(n) = \sum_{d \leq D } \frac{\boldsymbol{\mu}(d)}{d^k} \sum_{\substack{\nu \mod{d^k} \\ f(\nu) \equiv 0 \mod{d^k}} } \BEu{\substack{ 1\leq n\leq N \\ n \equiv \nu \mod{d^k} }} a(n) +  O(D^{-k+1+\delta})+O(N^{-\delta^2}+N^{-\Delta}), 
\end{equation}
for any $1\leq D\leq N^{1-\delta}$. 

We fix $D$ first. By \eqref{eqn_BR_pntap}, for any $1\leq d\leq D$  we have
\begin{equation}
	\lim_{N\to\infty} \BEu{\substack{ 1\leq n\leq N \\ n \equiv \nu \mod{d^k} }} a(n) =\alpha.
\end{equation}
Taking $N\to\infty$ in \eqref{thm_general_pf1} gives 
\begin{equation} \label{thm_general_pf2}
		\lim_{N\to\infty}\frac1N\sum_{1\leq n\leq N}\nu_k(f(n))a(n) = \alpha\cdot\sum_{d \leq D } \frac{\boldsymbol{\mu}(d)\rho_f(d^k)}{d^k}  +  O(D^{-k+1+\delta}). 
\end{equation}
Then \eqref{thm_general_eqn} follows by taking $D\to\infty$ in \eqref{thm_general_pf2}. This completes the proof of Theorem~\ref{thm_general}.
\end{proof}

\begin{proof}[Proof of Theorem~\ref{mainthm}]
If $f(x)=x^d+c$ with $k \geq (5d + 3)/9$ and $c\in\Z$, then by \cite[p. 180]{HB2013} \eqref{thm_general_cond} holds. So \eqref{mainthm_eqn} follows by Theorem~\ref{thm_general}. If $f(x)\in \Z[x]$  is an irreducible polynomial of degree $d\ge3$ and $k \geq 3d/4 + 1/4$, then by \cite[p. 145]{Browning2011} \eqref{thm_general_cond} still holds, and \eqref{mainthm_eqn} follows by Theorem~\ref{thm_general} again. 	
\end{proof}

\section{A result on reducible polynomials}\label{sec_reducible}

In this section, we consider a reducible integral polynomial studied by Booker and Browning in \cite{BookerBrowning2016}. Let $f\in\Z[x]$ be a non-constant squarefree polynomial. Assume that $f$ has no fixed prime divisor, by which we mean that there is no prime $p$ such that $p |f(n)$ for all $n\in \Z$. Assume that  each irreducible factor of $f$ has degree at most $3$. In 2016, Booker and Browning \cite[Theorem~1.2]{BookerBrowning2016} derived the following asymptotic formula for integers $n$ such that $f(n)$ is squarefree:
\begin{equation}\label{BookerBrowning2016}
	\sum_{\substack{n\leq N}} \boldsymbol{\mu}^2(f(n))  =N \prod_p
\left(1-\frac{\rho_f(p^2)}{p^2}\right) +O\!\left(\frac{N}{\log N}\right),
\end{equation}
where $\rho_f(q)$ is defined in \eqref{dfn_rho}.

To establish \eqref{BookerBrowning2016}, they \cite[(1)]{BookerBrowning2016} adapted the work of Hooley \cite[Chapter 4]{Hooley1976} and  Reuss \cite{Reuss2015blms}, showing that
\begin{equation}\label{BB_err}
	\#\set{n\leq N\colon \exists \, p> N^\delta \text{ s.t. } p^2\mid f(n)} = O(N^{1-\eta})
\end{equation}
for some $0<\delta<1/11$ and $\eta>0$. Using \eqref{BB_err} and the work of Hooley \cite[Chapter 4]{Hooley1976} (see also \cite{Hooley1967}), we will show that \eqref{conjecture} holds for such $f(n)$ and $k=2$.

\begin{theorem}\label{thm_BR_BB}
 Let $f\in\Z[x]$ be a non-constant squarefree polynomial with no fixed prime divisor. 	Assume that  each irreducible factor of $f$ has degree at most $3$. Let $(X,\mu, T)$ be a uniquely ergodic topological dynamical system, then we have
	 	\begin{equation}\label{thm_BR_BB_eqn}
		\lim_{N\to\infty}\frac1N\sum_{1\leq n\leq N} \boldsymbol{\mu}^2(f(n)) g(T^{\Omega(n)}x) = \prod_p\Big(1-\frac{\rho_f(p^2)}{p^2}\Big)\cdot \int_X g\,d\mu
	\end{equation}
for any $g\in C(X)$ and $x\in X$.

\end{theorem}

\begin{proof}

Put $a(n)=g(T^{\Omega(n)}x)$ and $\alpha=\int_X g\,d\mu$. Then we have \eqref{eqn_BR_pntap}, and we may assume that  $|a(n)|\leq 1$ for all $n$. Let $\xi=\frac16\log N$. Then 
\begin{equation}\label{BB_total}
	\sum_{\substack{1\leq n \leq N \\ f(n) \text{ is squarefree}}} a(n) = \sum_{\substack{1\leq n \leq N \\ p^2 \nmid f(n),\forall p\leq \xi }} a(n) + O(\Sigma_1+\Sigma_2),
\end{equation}
where
\begin{equation}
	\Sigma_1=\#\set{n\leq N\colon \exists \, p\in (\xi, N^\delta] \text{ s.t. } p^2\mid f(n)}, \Sigma_2=\#\set{n\leq N\colon \exists \, p> N^\delta \text{ s.t. } p^2\mid f(n)},
\end{equation}
and $0<\delta<1/11$ is chosen to be as in \eqref{BB_err}.
	
By \cite[(127)]{Hooley1976}, we have
\begin{equation}
\Sigma_1=O\of{\frac{N}{\log N}}.
\end{equation}

By \eqref{BB_err}, we have
\begin{equation}
	\Sigma_2=O(N^{1-\eta})
\end{equation}
for some $\eta>0$.

Now, we estimate the first summation in \eqref{BB_total}. Let $l_1$ indicate either 1 or squarefree number composed entirely of prime factors not exceeding $\xi$. Then
\begin{align}
	\sum_{\substack{1\leq n \leq N \\ p^2 \nmid f(n),\forall p\leq \xi }} a(n) & = \sum_{1\leq n\leq N} \Big(\sum_{l_1^2 \mid f(n)} \boldsymbol{\mu}(l_1) \Big)a(n) \nonumber\\
	&=\sum_{l_1} \boldsymbol{\mu}(l_1) \sum_{\substack{1\leq n \leq N \\ l_1^2 \mid f(n)}} a(n) \nonumber\\
	&=\sum_{l_1} \boldsymbol{\mu}(l_1) \sum_{\substack{\nu \mod{l_1^2} \\ f(\nu) \equiv 0 \mod{l_1^2} }} \sum_{\substack{1\leq n \leq N \\ n\equiv \nu \mod{l_1^2} }} a(n) \nonumber\\
	&=\sum_{l_1} \boldsymbol{\mu}(l_1) \sum_{\substack{\nu \mod{l_1^2} \\ f(\nu) \equiv 0 \mod{l_1^2} }} \Big(\frac{N}{l_1^2} +O(1)\Big)\BEu{\substack{1\leq n \leq N \\ n\equiv \nu \mod{l_1^2} }} a(n)   \nonumber\\
	&=N \sum_{l_1} \frac{\boldsymbol{\mu}(l_1) }{l_1^2}  \sum_{\substack{\nu \mod{l_1^2} \\ f(\nu) \equiv 0 \mod{l_1^2} }} \BEu{\substack{1\leq n \leq N \\ n\equiv \nu \mod{l_1^2} }} a(n) + O\of{\sum_{l_1}|\boldsymbol{\mu}(l_1)|\rho_f(l_1^2)}. \label{BB_main_term}
\end{align}

By \cite[(5)]{BookerBrowning2016}, we have $|\boldsymbol{\mu}(l_1)|\rho_f(l_1^2)\ll l_1^\ve$ for any $\ve>0$.  Moreover, by \cite[(95)]{Hooley1976} we have $l_1\leq N^{1/3}$. It follows that 
\begin{equation}\label{BB_err2}
	\sum_{l_1}|\boldsymbol{\mu}(l_1)|\rho_f(l_1^2) \ll N^{1/3+\ve}.
\end{equation}

Combining \eqref{BB_total}-\eqref{BB_err2} together gives us that
\begin{equation}\label{BB_keyeqn}
	\frac1N\sum_{1\leq n\leq N} \boldsymbol{\mu}^2(f(n))a(n) = \sum_{l_1} \frac{\boldsymbol{\mu}(l_1) }{l_1^2}  \sum_{\substack{\nu \mod{l_1^2} \\ f(\nu) \equiv 0 \mod{l_1^2} }} \BEu{\substack{1\leq n \leq N \\ n\equiv \nu \mod{l_1^2} }} a(n) + O\of{\frac1{\log N}}.
\end{equation}

Let $H\leq \xi$, then
\begin{equation}
	\sum_{l_1} \frac{\boldsymbol{\mu}(l_1) }{l_1^2}  \sum_{\substack{\nu \mod{l_1^2} \\ f(\nu) \equiv 0 \mod{l_1^2} }} \BEu{\substack{1\leq n \leq N \\ n\equiv \nu \mod{l_1^2} }} a(n) = \sum_{l_1 \leq H } + \sum_{l_1 >H} \colonequals S_3+S_4.
\end{equation}
Moreover, by $|\boldsymbol{\mu}(l_1)|\rho_f(l_1^2)\ll l_1^\ve$ again,
\begin{equation}
	S_4\ll  \sum_{l_1 >H} \frac{|\boldsymbol{\mu}(l_1)|\rho_f(l_1^2)}{l_1^2} \ll  \sum_{l_1 >H} \frac{l_1^\ve}{l_1^2} \ll \frac{1}{H^{1-\ve}}. 
\end{equation}

This implies that
\begin{equation}\label{final_eqn}
	\frac1N\sum_{1\leq n\leq N} \boldsymbol{\mu}^2(f(n))a(n) = \sum_{l_1\leq H} \frac{\boldsymbol{\mu}(l_1) }{l_1^2}  \sum_{\substack{\nu \mod{l_1^2} \\ f(\nu) \equiv 0 \mod{l_1^2} }} \BEu{\substack{1\leq n \leq N \\ n\equiv \nu \mod{l_1^2} }} a(n) +O(H^{-1+\ve})+  O\of{\frac1{\log N}}\end{equation}
for any $H\leq \xi=\frac16\log N$.

For fixed $H\leq \xi$, by \eqref{eqn_BR_pntap}, we have
\begin{equation}
	\lim_{N\to\infty} \BEu{\substack{1\leq n \leq N \\ n\equiv \nu \mod{l_1^2} }} a(n)  = \alpha
\end{equation}
for any $l_1\leq H$. Similar to the proof of Theorem~\ref{mainthm}, \eqref{thm_BR_BB_eqn} follows immediately by  taking $N\to\infty$ and then $H\to\infty$ in \eqref{final_eqn}. This completes the proof of Theorem~\ref{thm_BR_BB}.
\end{proof}

Take $f(n)=(n^2+1)(n^2+2)$ in \eqref{thm_BR_BB_eqn}. Then by Theorem~\ref{thm_BR_BB} and \cite[Theorem~2.1]{Dimitrov2021}, we obtain the following result.
\begin{corollary}
	Let $(X,\mu, T)$ be a uniquely ergodic topological dynamical system. Then we have
	 	\begin{equation}
		\lim_{N\to\infty}\frac1N\sum_{\substack{1\leq n \leq N\\ n^2+1,\, n^2+2 \text{ squarefree}}}  g(T^{\Omega(n)}x) = \prod_{p>2}\Big(1-\frac{(-1/p)+(-2/p)+2}{p^2}\Big)\cdot \int_X g\,d\mu
	\end{equation}
for any $g\in C(X)$ and $x\in X$.
\end{corollary}

\section{ABC implies equation \eqref{conjecture}}\label{sec_ABC}

In this section, we consider the squarefree values of general polynomials. Let $r\ge2$ and let $f(x)\in \Z[x]$ be a polynomial of degree $r$. Define $G_f\colonequals \gcd(f(n):n\geq1)$. Assume $f$ has no repeated factors and $G_f$ is squarefree. In 1998, under the ABC conjecture, Granville \cite{Granville1998} showed the asymptotic formula
\begin{equation}\label{Granville1998}
	\#\set{n\leq N \colon f(n) \text{ is squarefree}} \sim  c_f\cdot N
\end{equation}
for $c_f\colonequals\prod_p\Big(1-\frac{\rho_f(p^2)}{p^2}\Big)>0$, where $\rho_f(p^2)$ is defined by \eqref{dfn_rho}. Then Lee and Murty \cite{LeeMurty2007} provided an error term for \eqref{Granville1998} under both the ABC conjecture and the abscissa conjecture. In 2014, Murty and Pasten \cite{MurtyPasten2014} obtained an error term for \eqref{Granville1998} under only the ABC conjecture. In the following, under the assumption that the ABC conjecture is true, we will explain how to use Murty and Pasten's results to obtain \eqref{conjecture} for squarefree numbers.

\begin{theorem}\label{mainthm_ABC}
Let $r\ge2$. Let $f(x)\in \Z[x]$ be a polynomial of degree $r$, without repeated factors, and with $G_f$ squarefree. Let $(X,\mu, T)$ be a uniquely ergodic topological dynamical system. Then assuming the ABC conjecture, we have
\begin{equation}\label{mainthm_ABC_eqn}
		\lim_{N\to\infty}\frac1N\sum_{\substack{1\leq n\leq N \\ f(n)  \text{ is squarefree} }}  g(T^{\Omega(n)}x) = c_f \cdot \int_X g\,d\mu
	\end{equation}
for any $g\in C(X)$ and $x\in X$, where $c_f$ is defined  as in \eqref{Granville1998}.

\end{theorem}
\begin{proof}

The argument is similar to the proof of Theorem~\ref{thm_BR_BB}. We put $a(n)=g(T^{\Omega(n)}x)$ and $\alpha=\int_X g\,d\mu$. Let $\xi=\frac16\log N$. Then 
\begin{equation}\label{ABC_total}
	\sum_{\substack{1\leq n \leq N \\ f(n) \text{ is squarefree}}} a(n) = \sum_{\substack{1\leq n \leq N \\ p^2 \nmid f(n),\forall p\leq \xi }} a(n) + O(\Sigma_3+\Sigma_4),
\end{equation}
where
\begin{equation}
	\Sigma_3=\#\set{n\leq N\colon \exists \, p\in (\xi, N] \text{ s.t. } p^2\mid f(n)}, \Sigma_4=\#\set{n\leq N\colon \exists \, p> N \text{ s.t. } p^2\mid f(n)}.
\end{equation}
	
By \cite[Lemma 5.4]{MurtyPasten2014}, we have the same error term as $\Sigma_1$:
\begin{equation}
\Sigma_3=O\of{\frac{N}{\log N}}.
\end{equation}

By \cite[\S6(2)]{MurtyPasten2014}, which is proved under the ABC conjecture, we have a different error term from $\Sigma_2$:
\begin{equation}
	\Sigma_4=O\of{\frac{N}{(\log N)^\gamma}},
\end{equation}
where $0<\gamma<1$ is a computable constant that depends only on $r$. 

As regards the summation in \eqref{ABC_total} over $1\leq n \leq N $ with $p^2 \nmid f(n),\forall p\leq \xi$, the estimates \eqref{BB_main_term} and \eqref{BB_err2} still hold. This implies that 
\begin{equation} \label{ABC_keyeqn}
	\frac1N\sum_{1\leq n\leq N} \boldsymbol{\mu}^2(f(n))a(n) = \sum_{l_1 \mid \prod_{p\leq \xi}p} \frac{\boldsymbol{\mu}(l_1) }{l_1^2}  \sum_{\substack{\nu \mod{l_1^2} \\ f(\nu) \equiv 0 \mod{l_1^2} }} \BEu{\substack{1\leq n \leq N \\ n\equiv \nu \mod{l_1^2} }} a(n) + O\of{\frac1{(\log N)^\gamma}}.
\end{equation}

We see that \eqref{ABC_keyeqn} only differs from \eqref{BB_keyeqn} by an error term $O((\log N)^{-\gamma})$. This error term vanishes as $N\to\infty$.  Thus, by the remaining argument after \eqref{BB_keyeqn} in the proof of Theorem~\ref{thm_BR_BB}, we can get \eqref{mainthm_ABC_eqn} from \eqref{ABC_keyeqn}. This completes the proof of Theorem~\ref{mainthm_ABC}.
\end{proof}

\begin{remark}
	Combining the proof of Theorem~\ref{thm_BR_BB} and the ideas in the work \cite{Hooley1967} of Hooley, one can also show that \eqref{conjecture} holds for any irreducible integral polynomial $f(n)$ of degree $d\ge3$ and $k=d-1$. Moreover, 	Theorems~\ref{mainthm} and  \ref{thm_BR_BB} still hold if $g(T^{\Omega(n)}x)$ is replaced by $g(T^{\Omega(mn+r)}x)$ for any $g\in C(X)$, $x\in X$, $m\in \N$ and $r\in \set{0,1\dots,m-1}$. By \cite[Corollary 1.6]{Xiao2024}, they also hold for $g(T^{\Omega(\lfloor \alpha n +\beta\rfloor)}x)$ along Beaty sequences $\lfloor \alpha n +\beta\rfloor$, where $\alpha>0, \beta\in \R$ and $\alpha+\beta>1$. Here $\lfloor t \rfloor$ is the largest integer such that $0\leq t-\lfloor t \rfloor <1$.
\end{remark}

\begin{remark} Let $a,b: \N \to \C$ be two sequences, they are called \textit{asymptotically uncorrelated} if 
\[
\lim_{N\to\infty}\left(\AVG a(n) \overline{b(n)} - \bigg( \AVG a(n) \bigg)\bigg( \AVG \overline{b(n)} \bigg) \right)=0.
\]
For example, in 2010 Sarnak \cite{Sarnak2010ias,Sarnak2010} conjectured that any deterministic sequence is asymptotically uncorrelated with the M\"obius function. In \cite{Loyd2023}, Loyd showed that the sequences capturing the behaviors of the Erd\H{o}s-Kac theorem and Bergelson-Richter’s theorem are asymptotically uncorrelated. This leads to explore what arithmetic functions $h:\N\to\C$ satisfy the following disjointness with $g(T^{\Omega(n)}x)$ in \eqref{eqn_BR2022thmA}:
\begin{equation}\label{question}
		\lim_{N\to\infty}\frac1N\sum_{n=1}^N h(n)g(T^{\Omega(n)}x)=\lim_{N\to\infty}\frac1N\sum_{n=1}^N h(n)\cdot\int_X g \,d\mu
\end{equation}
for any $g\in C(X)$ and $x\in X$, provided the mean of $h(n)$ exists. In \cite{Wang2022,WWYY2023, Wang2025jnt, Wang2025bams, WangYi2025, LWWY2025}, such kind of disjointness are established for various arithmetic functions related to the largest prime factors of integers, $k$-free numbers, $k$-full numbers, $\Omega(n)-\omega(n)$ and squarefull kernel functions, where $\omega(n)=\sum_{p\mid n}1$. One may expect \eqref{question} to hold for more functions. This will give a series of generalizations of the PNT in this kind of disjointness form.
\end{remark}

\section*{Acknowledgments}
The authors are deeply grateful to the referee for a very careful review and a number of insightful comments and suggestions, which improve this paper a lot; in particular, for the suggestion of using the results of Murty and Pasten in \cite{MurtyPasten2014} to deduce Theorem~\ref{mainthm_ABC}.  Biao Wang is supported by the National Natural Science Foundation of China (Grant No.~12561001). Shaoyun Yi is supported by the National Natural Science Foundation of China (Nos.~12301016, 12471187) and the Fundamental Research Funds for the Central Universities (No.~20720230025).

\end{document}